\documentclass[reqno]{amsart}
\usepackage{amsmath}

\usepackage{amsfonts}
\usepackage{amssymb}
\usepackage{hyperref}

\usepackage[a4paper,top=3.5cm,bottom=3cm,left=2cm,right=2cm]{geometry}

\begin{document}
\title[On the existence and uniqueness of solution]{On the existence and uniqueness of solution for fractional differential equations with nonlocal multi-point boundary conditions}
\author[F. Haddouchi]{Faouzi Haddouchi}
\address{
Department of Physics, University of Sciences and Technology of
Oran-MB, El Mnaouar, BP 1505, 31000 Oran, Algeria
\newline
And
\newline
Laboratory of Fundamental and Applied Mathematics of Oran (LMFAO)\\
University of Oran 1 Ahmed Benbella, 31000 Oran,
Algeria}
\email{fhaddouchi@gmail.com}
\begin{abstract}
This paper presents some sufficient conditions for the existence of solutions of fractional differential equation with nonlocal multi-point boundary conditions involving Caputo fractional derivative and integral boundary conditions. Our analysis relies on the Banach contraction principle, Boyd and Wong fixed point theorem, Leray-Schauder nonlinear alternative. Finally, examples are provided to illustrate our main results.

\end{abstract}

\maketitle \numberwithin{equation}{section}
\newtheorem{theorem}{Theorem}[section]
\newtheorem{lemma}[theorem]{Lemma}
\newtheorem{definition}[theorem]{Definition}
\newtheorem{proposition}[theorem]{Proposition}
\newtheorem{corollary}[theorem]{Corollary}
\newtheorem{remark}[theorem]{Remark}
\newtheorem{exmp}{Example}[section]

\section{Introduction}\label{sec1}

Differential equations with fractional order have are powerful tool for solving
practical problems that arise in fields like as control theory, chemical physics, economics, heat conduction, viscoelasticity, biological science, ecology, aerodynamics, etc., see for example, \cite{Meral2010}, \cite{Oldham2010}, \cite{Nigmatullin2010}, \cite{Orsingher2004}, \cite{Kilbas2006}. In particular, the
book by Oldham and Spanier \cite{Oldham1974} had a chronological listing on major works in the study of
fractional calculus.

In the recent years, there has been a significant development in ordinary and partial differential equations involving fractional derivatives, see the monographs of Kilbas et al. \cite{Kilbas2006}, Miller and Ross \cite{Miller1993}, \cite{Podlubny1999}.

By the use of techniques of nonlinear analysis, many authors have studies the existence and uniqueness of solutions of nonlinear fractional differential equations with a variety boundary conditions as special cases because they can accurately describe the actual phenomena. They include two-point, three-point, multi-point and nonlocal boundary value problems with integral boundary conditions as special cases, see \cite{Agarwal2017, Ahmad2015, Ahmad2016, Ahmad2016(2), Di2018, Guo2008, Sun2014, Yang2016, Zhang2017, Zhang2012, Zhang2013, Zhang(2)2012, Zhao2015, Zhao(2)2015, Zhao2016, Zhao2017, Zhao(2)2016, Haddouchi2018} and references therein.

Integral boundary conditions are encountered in various applications such as population dynamics, blood flow models, chemical engineering, cellular systems, underground water flow, heat transmission, plasma physics, thermoelasticity, etc.

Many results can be found in the literature concerning multi-point boundary value problems with integral conditions for differential equations of fractional order \cite{Agarwal(2)2017, Henderson2017, Liu2012, Li2017, Wang2018, Jia2012}.

For instance, in \cite{Li2017}, Li and Qi discussed the existence of multiple
positive solutions to the following fractional boundary value problems

\begin{equation*}
\begin{cases}
^{c}D^{\alpha}x(t)+h(t)g(t,x(t))=0,\ t\in(0,1),
\\
x^{(i)}(0)=0,\ i=2,...,n-1,
\\
x^{\prime}(0)=\sum_{i=1}^{m-2}b_{i}x^{\prime}(\xi_{i}), \ x(1)=\sum_{i=1}^{m-2}a_{i}x(\xi_{i}),
\end{cases}
\end{equation*}
where $^{c}D^{\alpha}$ is the Caputo fractional derivatives, $n-1<\alpha\leq n$, $n\geq3$ is an integer, $a_{i}, b_{i}\geq0$, ($i=1,...,m-2$), $0\leq\sum_{i=1}^{m-2}a_{i}<1$ and $0\leq\sum_{i=1}^{m-2}b_{i}<1$, $0<\xi_{1}<\xi_{2}<...<\xi_{m-2}<1$, $m>2$ is an integer, $h$ and $g$ are a given continuous functions.
Using the five-functional fixed-point theorem, they obtained the existence of multiple positive solutions for the above boundary value problems.

In \cite{Wang2018}, Wang et al., investigated the following fractional differential equations that contain both the integral boundary condition and the multi-point boundary condition

\begin{equation*}
\begin{cases}
D^{\sigma}x(t)+f(t,x(t))=0,\ t\in[0,1],
\\
x^{(i)}(0)=0,\ i=0,1,...,n-2,
\\
x(1)=\sum_{i=1}^{m-2}\beta_{i}\int_{0}^{\eta_{i}}x(s)ds+\sum_{i=1}^{m-2}\gamma_{i}x(\eta_{i}),
\end{cases}
\end{equation*}
where $D^{\sigma}$ represents the standard Riemann-Liouville fractional derivative of order $\sigma$ satisfying $n-1<\sigma\leq n$ with $n\geq3$, $0<\eta_{1}<\eta_{2}<...<\eta_{m-2}<1$ and $\beta_{i}, \gamma_{i}>0$ with $1\leq i \leq m-2$, where $m\geq3$ is an integer. $f:[0,1]\times \mathbb{R}\rightarrow \mathbb{R}$ is a given continuous function.
By using suitable fixed point theorems, the authors obtained several existence and uniqueness results of
positive solutions.

In \cite{Jia2012}, Jia et al. investigated the existence and uniqueness of nontrivial solutions
to the following higher fractional differential equation

\begin{equation*}
\begin{cases}
-D^{\alpha}x(t)=f(t,x(t), D^{\mu_{1}}x(t), D^{\mu_{2}}x(t),...,D^{\mu_{n-1}}x(t) ), \ t\in(0,1),
\\
x(0)=0,\  D^{\mu_{i}}x(0)=0,\ D^{\mu}x(1)=\sum_{j=1}^{p-2}a_{j}D^{\mu}x(\xi_{j}),\    1\leq i \leq n-1,
\end{cases}
\end{equation*}
where $n\geq3$, $n\in \mathbb{N}$, $n-1<\alpha \leq n$, $n-l-1<\alpha-\mu_{l} \leq n-l$, for $l=1,2,...,n-2$, and $\mu-\mu_{n-1}>0$,  $\alpha-\mu_{n-1}\leq2$,  $\alpha-\mu>1$, $a_{j}\geq 0$, $0<\xi_{1}<\xi_{2}<...<\xi_{p-2}<1$, $\sum_{j=1}^{p-2}a_{j}\xi_{j}^{\alpha-\mu-1}\neq1$, $D^{\alpha}$ is
the standard Riemann-Liouville derivative, and $f :[0,1]\times \mathbb{R}^{n}\rightarrow \mathbb{R} $ is continuous.
The existence results are obtained with the aid of some classical fixed point theorems.

More recently, Agarwal et al. \cite{Agarwal(2)2017} studied the following fractional order boundary value problem

\begin{equation*}
\begin{cases}
^{c}D^{q}x(t)=f(t,x(t)),\ 1<q\leq 2, \ t\in[0,1],
\\
x(0)=\delta x(\sigma), \ a\ ^{c}D^{p}x(\zeta_{1})+b\ ^{c}D^{p}x(\zeta_{2})=\sum_{i=1}^{m-2}\alpha_{i}x(\beta_{i}), \ 0<p<1.
\end{cases}
\end{equation*}

Together with the above fractional differential equation they also investigated the boundary conditions

\begin{equation*}
x(0)=\delta_{1}\int_{0}^{\sigma}x(s)ds, \ {a}\ ^{c}D^{p}x(\zeta_{1})+{b}\  ^{c}D^{p}x(\zeta_{2})=\sum_{i=1}^{m-2}\alpha_{i}x(\beta_{i}), \ 0<p<1,
\end{equation*}
where $^{c}D^{q}$,$^{c}D^{p}$ denote the Caputo fractional derivatives of orders $q, p$ and $f:[0,1]\times\mathbb{R}\longrightarrow\mathbb{R}$ is a given continuous
function and $\delta, \delta_{1}, a, b, \alpha_{i}\in \mathbb{R}$, with $0<\sigma<\zeta_{1}<\beta_{1}<\beta_{2}<...<\beta_{m-2}<\zeta_{2}<1.$ \\
The existence and uniqueness results were proved via some well known tools of the fixed point theory.

Motivated by the above works, in this paper, we investigate the following Caputo fractional differential equation

\begin{equation} \label{eq-1.1}
^{c}D^{q}x(t)=f(t,x(t)),\ t\in[0,1],
\end{equation}
with nonlocal multi-point Caputo fractional derivative and integral boundary conditions
\begin{equation} \label{eq-1.2}
\begin{cases}
^{c}D^{\sigma}x(\xi)=\sum_{i=1}^{n}\alpha_{i}{^{c}D^{\nu}}x(\eta_{i}),
\\
 x(1)=\sum_{i=1}^{n}\beta_{i}\int_{0}^{\eta_{i}}x(s)ds+\sum_{i=1}^{n}\gamma_{i}x(\eta_{i}),
\end{cases}
\end{equation}

where $^{c}D^{\mu}$ is the Caputo fractional derivative of order $\mu\in\{q, \sigma, \nu\}$ such that $1<q\leq2$, $0<\sigma,\nu\leq1$, and $f:[0,1]\times\mathbb{R}\longrightarrow\mathbb{R}$ is a given continuous function, $0<\xi<\eta_{1}<\eta_{2}<...<\eta_{n}<1$ and $\alpha_{i}$, $\beta_{i}$, $\gamma_{i} $, $i=1,\dots,n$ are appropriate real constants.

\section{Preliminaries}\label{sec2}

In this section, we recall some basic definitions of fractional calculus and an auxiliary lemma to define the solution for the problem \eqref{eq-1.1}-\eqref{eq-1.2} is presented.

\begin{definition}
The Riemann-Liouville fractional integral of order $q$ for a continuous function $f$ is defined as
\begin{equation*}
I^{q}f(t)=\frac{1}{\Gamma(q)}\int_{0}^{t}\frac{f(s)}{(t-s)^{1-q}}ds, \ q>0,
\end{equation*}
provided the integral exists, where $\Gamma(.)$ is the gamma function, which is defined by $\Gamma(x)=\int_{0}^{\infty}t^{x-1}e^{-t}dt$.
\end{definition}

\begin{definition}
For at least n-times continuously differentiable function $f:[0,\infty)\rightarrow \mathbb{R}$, the Caputo derivative of fractional order $q$ is defined as
\begin{equation*}
^{c}D^{q}f(t)=\frac{1}{\Gamma(n-q)}\int_{0}^{t}\frac{f^{(n)}(s)}{(t-s)^{q+1-n}}ds,\ n-1<q<n,\ n=[q]+1,
\end{equation*}
where $[q]$ denotes the integer part of the real number $q$.
\end{definition}

\begin{lemma}[\cite{Kilbas2006}] \label{lem 2.1}
For $q > 0$, the general solution of the fractional differential equation $^{c}D^{q}x(t)=0$ is
given by
\begin{equation*}
x(t)=c_{0}+c_{1}t+...+c_{n-1}t^{n-1},
\end{equation*}
where $c_{i}\in \mathbb{R}$, $i = 0, 1,...,n-1 \ (n = [q] + 1)$.
\end{lemma}
According to Lemma \ref{lem 2.1}, it follows that
\begin{equation*}
{I^{q}}\ {^{c}D^{q}}x(t)=x(t)+c_{0}+c_{1}t+...+c_{n-1}t^{n-1},
\end{equation*}
for some $c_{i}\in \mathbb{R}$, $i = 0, 1,...,n-1 \ (n = [q] + 1)$.

\begin{lemma}[\cite{Podlubny1999}, \cite{Kilbas2006}]\label{lem 2.2}
If $\beta> \alpha>0$ and $x \in L_{1}[0,1]$, then
\item[(i)] ${^{c}D^{\alpha}}\ {I^{\beta}}x(t)={I^{\beta-\alpha}}x(t)$, holds almost everywhere on $[0,1]$ and it is valid at any point $t\in [0,1]$ if $x\in C[0,1]$;\  ${^{c}D^{\alpha}}\ {I^{\alpha}}x(t)=x(t) $, for all \ $t\in[0,1]$.
\item[(ii)] ${^{c}D^{\alpha}}t^{\lambda-1}=\frac{\Gamma(\lambda)}{\Gamma(\lambda-\alpha)}t^{\lambda-\alpha-1}$,
$\lambda>[\alpha]$ \ and \ ${^{c}D^{\alpha}}t^{\lambda-1}=0$, \  $\lambda<[\alpha]$.
\end{lemma}

\begin{lemma}\label{lem 2.3}
Let $h:[0,1]\longrightarrow\mathbb{R}$ be a continuous function. Then the solution of the linear fractional differential equation
\begin{equation} \label{eq-2.1}
^{c}D^{q}x(t)=h(t), 1<q\leq2,\ t\in[0,1],
\end{equation}
supplemented with boundary conditions \eqref{eq-1.2} is equivalent to the integral equation

\begin{eqnarray} \label{eq-2.3}
x(t)&=&\int_{0}^{t}\frac{(t-s)^{q-1}}{\Gamma(q)}h(s)ds+\frac{1}{\Delta_{2}}\Bigg(-\int_{0}^{1}\frac{(1-s)^{q-1}}{\Gamma(q)}h(s)ds +\sum_{i=1}^{n}\beta_{i}\int_{0}^{\eta_{i}}\frac{(\eta_{i}-s)^{q}}{\Gamma(q+1)}h(s)ds \nonumber\\ &&+\sum_{i=1}^{n}\gamma_{i}\int_{0}^{\eta_{i}}\frac{(\eta_{i}-s)^{q-1}}{\Gamma(q)}h(s)ds\Bigg) +\frac{\Gamma(2-\sigma)\Gamma(2-\nu)\Big(\Delta_{3}+2\Delta_{2}t\Big)}{2\Delta_{1}\Delta_{2}}\Bigg(-\int_{0}^{\xi}\frac{(\xi-s)^{q-\sigma-1}}{\Gamma(q-\sigma)}h(s)ds\nonumber\\
&&+\sum_{i=1}^{n}\alpha_{i}\int_{0}^{\eta_{i}}\frac{(\eta_{i}-s)^{q-\nu-1}}{\Gamma(q-\nu)}h(s)ds \Bigg),
\end{eqnarray}
where
\begin{align} \label{eq-2.4}
\Delta_{1}&=\xi^{1-\sigma}\Gamma(2-\nu)-\Gamma(2-\sigma)\sum_{i=1}^{n}\alpha_{i}\eta_{i}^{1-\nu}\neq0 \nonumber\\
\Delta_{2}&=1-\sum_{i=1}^{n}(\beta_{i}\eta_{i}+\gamma_{i})\neq0\\
\Delta_{3}&=-2+\sum_{i=1}^{n}\eta_{i}(\beta_{i}\eta_{i}+2\gamma_{i}). \nonumber
\end{align}

\end{lemma}
\begin{proof}
From Lemma \ref{lem 2.1}, we may reduce \eqref{eq-2.1} to an equivalent integral equation
\begin{equation} \label{eq-2.5}
x(t)=I^{q}h(t)+c_{0}+c_{1}t,
\end{equation}
where $c_{0}, \ c_{1}\in \mathbb{R}$ are arbitrary constants. Consequently, the general solution of equation \eqref{eq-2.1} is
\begin{equation} \label{eq-2.6}
x(t)=\frac{1}{\Gamma(q)}\int_{0}^{t}(t-s)^{q-1}h(s)ds+c_{0}+c_{1}t,
\end{equation}
and
\begin{eqnarray} \label{eq-2.7}
\int_{0}^{\eta_{i}}x(s)ds&=&\int_{0}^{\eta_{i}}\Bigg(\int_{0}^{s}\frac{(s-\tau)^{q-1}}{\Gamma(q)}h(\tau)d\tau\Bigg)ds+c_{0}\eta_{i}
+c_{1}\frac{\eta_{i}^{2}}{2} \nonumber\\
&=&\int_{0}^{\eta_{i}}\frac{(\eta_{i}-s)^{q}}{\Gamma(q+1)}h(s)ds+c_{0}\eta_{i}+c_{1}\frac{\eta_{i}^{2}}{2},\ i=1,2,...,n.
\end{eqnarray}
Now, in view of Lemma \ref{lem 2.2}, by taking the Caputo fractional derivative of order $\nu$ and $\sigma$ to both sides of \eqref{eq-2.6}, we get
\begin{align}
^{c}D^{\nu}x(\eta_{i})&=\int_{0}^{\eta_{i}}\frac{(\eta_{i}-s)^{q-\nu-1}}{\Gamma(q-\nu)}h(s)ds+
c_{1}\frac{\eta_{i}^{1-\nu}}{\Gamma(2-\nu)},\ i=1,2,...,n.\nonumber\\
^{c}D^{\sigma}x(\xi)&=\int_{0}^{\xi}\frac{(\xi-s)^{q-\sigma-1}}{\Gamma(q-\sigma)}h(s)ds+
c_{1}\frac{\xi^{1-\sigma}}{\Gamma(2-\sigma)}\nonumber.
\end{align}

The boundary condition $^{c}D^{\sigma}x(\xi)=\sum_{i=1}^{n}\alpha_{i}{^{c}D^{\nu}}x(\eta_{i})$ implies that
\begin{equation} \label{eq-2.8}
\int_{0}^{\xi}\frac{(\xi-s)^{q-\sigma-1}}{\Gamma(q-\sigma)}h(s)ds+
c_{1}\frac{\xi^{1-\sigma}}{\Gamma(2-\sigma)}\\
=\sum_{i=1}^{n}\alpha_{i}\int_{0}^{\eta_{i}}\frac{(\eta_{i}-s)^{q-\nu-1}}{\Gamma(q-\nu)}h(s)ds+
\frac{c_{1}}{\Gamma(2-\nu)}\sum_{i=1}^{n}\alpha_{i}\eta_{i}^{1-\nu}.
\end{equation}

which, on solving, yields

\begin{equation*}
c_{1}=\frac{\Gamma(2-\sigma)\Gamma(2-\nu)}{\Delta_{1}}\Bigg(-\int_{0}^{\xi}\frac{(\xi-s)^{q-\sigma-1}}{\Gamma(q-\sigma)}h(s)ds+\sum_{i=1}^{n}\alpha_{i}\int_{0}^{\eta_{i}}\frac{(\eta_{i}-s)^{q-\nu-1}}
{\Gamma(q-\nu)}h(s)ds\Bigg).
\end{equation*}

The second condition of \eqref{eq-1.1} implies that
\begin{eqnarray}\label{eq-2.8}
\int_{0}^{1}\frac{(1-s)^{q-1}}{\Gamma(q)}h(s)ds+c_{0}+c_{1}&=&\sum_{i=1}^{n}\beta_{i}\int_{0}^{\eta_{i}}\frac{(\eta_{i}-s)^{q}}
{\Gamma(q+1)}h(s)ds+c_{0}\sum_{i=1}^{n}\beta_{i}\eta_{i}+\frac{c_{1}}{2}\sum_{i=1}^{n}\beta_{i}\eta_{i}^{2} +c_{0}\sum_{i=1}^{n}\gamma_{i} \nonumber \\
&&+c_{1}\sum_{i=1}^{n}\gamma_{i}\eta_{i}+\sum_{i=1}^{n}\gamma_{i}\int_{0}^{\eta_{i}}\frac{(\eta_{i}-s)^{q-1}}
{\Gamma(q)}h(s)ds.
\end{eqnarray}

which, on inserting the value of $c_{1}$ in \eqref{eq-2.8}, we obtain

\begin{eqnarray}
c_{0} & = &\frac{1}{\Delta_{2}}\Bigg(-\int_{0}^{1}\frac{(1-s)^{q-1}}{\Gamma(q)}h(s)ds+\sum_{i=1}^{n}\beta_{i}\int_{0}^{\eta_{i}}\frac{(\eta_{i}-s)^{q}}
{\Gamma(q+1)}h(s)ds+ \sum_{i=1}^{n}\gamma_{i}\int_{0}^{\eta_{i}}\frac{(\eta_{i}-s)^{q-1}}
{\Gamma(q)}h(s)ds\Bigg)\nonumber \\
&&+\frac{\Delta_{3}\Gamma(2-\sigma)\Gamma(2-\nu)}{2\Delta_{1}\Delta_{2}}\Bigg(-\int_{0}^{\xi}\frac{(\xi-s)^{q-\sigma-1}}{\Gamma(q-\sigma)}h(s)ds
+\sum_{i=1}^{n}\alpha_{i}\int_{0}^{\eta_{i}}\frac{(\eta_{i}-s)^{q-\nu-1}}
{\Gamma(q-\nu)}h(s)ds\Bigg).
\end{eqnarray}

Substituting the values of $c_{0}$ and $c_{1}$ in \eqref{eq-2.6} we obtain the solution \eqref{eq-2.3}. This completes the proof.
\end{proof}

\section{Existences results}\label{sec3}

In this section, we establish sufficient conditions for the existence of solutions to the fractional order boundary value problem \eqref{eq-1.1}-\eqref{eq-1.2} using certain fixed point theorems.

Let $\mathcal{C}=C([0,1],\mathbb{R})$ be the Banach space of all continuous functions from $[0,1]$ into $\mathbb{R}$ endowed with the norm: $\|x\|=\sup\{|x(t)|, t\in[0,1 ]\}$. In view of Lemma \ref{lem 2.3}, we define an operator $\mathcal{A} : \mathcal{C}\rightarrow\mathcal{C}$ by
\begin{eqnarray} \label{eq-3.1}
(\mathcal{A}x)(t)&=&\int_{0}^{t}\frac{(t-s)^{q-1}}{\Gamma(q)}f(s,x(s))ds+\frac{1}{\Delta_{2}}\Bigg(-\int_{0}^{1}\frac{(1-s)^{q-1}}{\Gamma(q)}f(s,x(s))ds+\sum_{i=1}^{n}\beta_{i}\int_{0}^{\eta_{i}}\frac{(\eta_{i}-s)^{q}}{\Gamma(q+1)}f(s,x(s))ds \nonumber\\
&&+\sum_{i=1}^{n}\gamma_{i}\int_{0}^{\eta_{i}}\frac{(\eta_{i}-s)^{q-1}}{\Gamma(q)}f(s,x(s))ds\Bigg)
+\frac{\Gamma(2-\sigma)\Gamma(2-\nu)\Big(\Delta_{3}+2\Delta_{2}t\Big)}{2\Delta_{1}\Delta_{2}} \nonumber\\
&&\times\Bigg(-\int_{0}^{\xi}\frac{(\xi-s)^{q-\sigma-1}}{\Gamma(q-\sigma)}f(s,x(s))ds+\sum_{i=1}^{n}\alpha_{i}\int_{0}^{\eta_{i}}\frac{(\eta_{i}-s)^{q-\nu-1}}{\Gamma(q-\nu)}f(s,x(s))ds \Bigg),
\end{eqnarray}
with $\Delta_{1}\neq0$ and $\Delta_{2}\neq0$, defined by \eqref{eq-2.4}. Clearly, $x$ is a solution of problem \eqref{eq-1.1}-\eqref{eq-1.2} if and only if $x$ is a fixed point of the operator $\mathcal{A}$.

For the sake of convenience, in the sequel we set

\begin{align}
 \begin{split}
\Theta&=\frac{1}{\Gamma(q+1)}+\frac{1}{|\Delta_{2}|\Gamma(q+2)}\Bigg(1+q+\sum_{i=1}^{n}\eta_{i}^{q}\Big(\eta_{i}|\beta_{i}|+(q+1)|\gamma_{i}|\Big)\Bigg) \label{eq-3.2}\\
&\quad +\frac{\Gamma(2-\sigma)\Gamma(2-\nu)\Big(|\Delta_{3}|+2|\Delta_{2}|\Big)}{2|\Delta_{1}\Delta_{2}|}\Bigg(\frac{\xi^{q-\sigma}}{\Gamma(q-\sigma+1)}+
\sum_{i=1}^{n}|\alpha_{i}|\frac{\eta_{i}^{q-\nu}}{\Gamma(q-\nu+1)}\Bigg)
   \end{split} \\
\Omega&=\frac{1}{\Gamma(q)}+\frac{\Gamma(2-\sigma)\Gamma(2-\nu)}{|\Delta_{1}|}\Bigg(\frac{\xi^{q-\sigma}}{\Gamma(q-\sigma+1)}
+\sum_{i=1}^{n}|\alpha_{i}|\frac{\eta_{i}^{q-\nu}}{\Gamma(q-\nu+1)}\Bigg)\label{eq-3.3}
\end{align}

Now we prove an existence and uniqueness result via Banach's fixed point theorem.

\begin{theorem} \label{thm 3.1}
Let $f:[0,1]\times\mathbb{R}\longrightarrow\mathbb{R}$ be a continuous function satisfying the Lipschitz
condition:\\
$(H_{1})$\ there exists a constant $L>0$ such that $|f(t,x)-f(t,y)|\leq L|x-y|$, for each $t\in[0, 1]$
and $x,y\in \mathbb{R}$.\\
Then the fractional boundary value problem \eqref{eq-1.1}-\eqref{eq-1.2} has a unique solution on $[0,1]$ if
\begin{equation}\label{eq-3.4}
L \Theta<1,
\end{equation}
where $\Theta$ is given by \eqref{eq-3.2},
\end{theorem}
\begin{proof}
We transform the problem \eqref{eq-1.1}-\eqref{eq-1.2} into a fixed point problem $x=\mathcal{A}x$, where the operator $\mathcal{A}$ is defined by \eqref{eq-3.1}. Applying Banach's contraction mapping principle, we shall show that $\mathcal{A}$ has a unique fixed point.
\newline
Setting $\sup\{|f(t,0)|, t\in[0,1]\}=M<\infty$ and choose a constant $\rho>0$ satisfying
\[\rho\geq \Theta{M}(1-\Theta{L})^{-1}.\]

First, we will show that $\mathcal{A}B_{\rho}\subset B_{\rho}$, where  $B_{\rho}=\{x\in \mathcal{C}:\|x\|\leq \rho\}$.
From $(H_{1})$, for any $x\in B_{\rho}$, and $t\in[0,1]$, we get
\begin{gather} \label{eq-3.5}
\begin{aligned}
|f(t,x(t))|& \leq |f(t,x(t))-f(t,0)|+|f(t,0)| \\
&\leq L \|x\|+M\\
&\leq L {\rho}+M.
\end{aligned}
\end{gather}
Using \eqref{eq-3.1}, \eqref{eq-3.5} and \eqref{eq-3.2}, we obtain
\begin{eqnarray*}
|(\mathcal{A}x)(t)|&\leq&\int_{0}^{t}\frac{(t-s)^{q-1}}{\Gamma(q)}|f(s,x(s))|ds+\frac{1}{|\Delta_{2}|}
\Bigg(\int_{0}^{1}\frac{(1-s)^{q-1}}{\Gamma(q)}|f(s,x(s))|ds\\
&&+\sum_{i=1}^{n}|\beta_{i}|\int_{0}^{\eta_{i}}\frac{(\eta_{i}-s)^{q}}{\Gamma(q+1)}|f(s,x(s))|ds+\sum_{i=1}^{n}|\gamma_{i}|\int_{0}^{\eta_{i}}\frac{(\eta_{i}-s)^{q-1}}{\Gamma(q)}|f(s,x(s))|ds
\Bigg) \nonumber\\
&&+\frac{\Gamma(2-\sigma)\Gamma(2-\nu)\Big(|\Delta_{3}|+2|\Delta_{2}|\Big)}{2|\Delta_{1}\Delta_{2}|}\Bigg(\int_{0}^{\xi}\frac{(\xi-s)^{q-\sigma-1}}{\Gamma(q-\sigma)}|f(s,x(s))|ds\nonumber \\
&&+\sum_{i=1}^{n}|\alpha_{i}|\int_{0}^{\eta_{i}}\frac{(\eta_{i}-s)^{q-\nu-1}}{\Gamma(q-\nu)}|f(s,x(s))|ds \Bigg)\\
&\leq&\Bigg\{\int_{0}^{t}\frac{(t-s)^{q-1}}{\Gamma(q)}ds+\frac{1}{|\Delta_{2}|}
\Bigg(\int_{0}^{1}\frac{(1-s)^{q-1}}{\Gamma(q)}ds \nonumber \\
&&+\sum_{i=1}^{n}|\beta_{i}|\int_{0}^{\eta_{i}}\frac{(\eta_{i}-s)^{q}}{\Gamma(q+1)}ds+\sum_{i=1}^{n}|\gamma_{i}|\int_{0}^{\eta_{i}}\frac{(\eta_{i}-s)^{q-1}}{\Gamma(q)}ds\Bigg)\\ &&+\frac{\Gamma(2-\sigma)\Gamma(2-\nu)\Big(|\Delta_{3}|+2|\Delta_{2}|\Big)}{2|\Delta_{1}\Delta_{2}|}\Bigg(\int_{0}^{\xi}\frac{(\xi-s)^{q-\sigma-1}}{\Gamma(q-\sigma)}ds\nonumber \end{eqnarray*}
\begin{eqnarray*}
&&+\sum_{i=1}^{n}|\alpha_{i}|\int_{0}^{\eta_{i}}\frac{(\eta_{i}-s)^{q-\nu-1}}{\Gamma(q-\nu)}ds \Bigg)\Bigg\}\times\big(L\|x\|+M\big)\\
&=&\Bigg\{\frac{t^{q}}{\Gamma(q+1)}+\frac{1}{|\Delta_{2}|}\Bigg(\frac{1}{\Gamma(q+1)}+\sum_{i=1}^{n}|\beta_{i}|\frac{\eta_{i}^{q+1}}{\Gamma(q+2)}
+\sum_{i=1}^{n}|\gamma_{i}|\frac{\eta_{i}^{q}}{\Gamma(q+1)}\Bigg)\nonumber\\
&&+\frac{\Gamma(2-\sigma)\Gamma(2-\nu)\Big(|\Delta_{3}|+2|\Delta_{2}|\Big)}{2|\Delta_{1}\Delta_{2}|}\Bigg(\frac{\xi^{q-\sigma}}{\Gamma(q-\sigma+1)}\nonumber \\
&&+\sum_{i=1}^{n}|\alpha_{i}|\frac{\eta_{i}^{q-\nu}}{\Gamma(q-\nu+1)} \Bigg)\Bigg\}\times\big(L\|x\|+M\big)\\
&\leq&\Bigg\{\frac{1}{\Gamma(q+1)}+\frac{1}{|\Delta_{2}|\Gamma(q+2)}\Bigg(1+q+\sum_{i=1}^{n}\eta_{i}^{q}\bigg(\eta_{i}|\beta_{i}|+(q+1)|\gamma_{i}|\bigg)\Bigg)\nonumber\\
&&+\frac{\Gamma(2-\sigma)\Gamma(2-\nu)\Big(|\Delta_{3}|+2|\Delta_{2}|\Big)}{2|\Delta_{1}\Delta_{2}|}\Bigg(\frac{\xi^{q-\sigma}}{\Gamma(q-\sigma+1)}\nonumber \\
&&+\sum_{i=1}^{n}|\alpha_{i}|\frac{\eta_{i}^{q-\nu}}{\Gamma(q-\nu+1)} \Bigg)\Bigg\}\times\big(L\|x\|+M\big)\\
&=&(L{\rho}+M)\Theta \leq\rho,
\end{eqnarray*}
which means that $\|\mathcal{A}x\|\leq \rho$. Therefore, we have, $\mathcal{A}B_{\rho}\subset B_{\rho}$.

Now, for  $x,y\in \mathcal{C}$, and for each $t\in[0,1]$, we have

\begin{eqnarray*}
|(\mathcal{A}x)(t)-(\mathcal{A}y)(t)|&\leq&\int_{0}^{t}\frac{(t-s)^{q-1}}{\Gamma(q)}|f(s,x(s))-f(s,y(s))|ds\nonumber\\
&&+\frac{1}{|\Delta_{2}|}\Bigg(\int_{0}^{1}\frac{(1-s)^{q-1}}{\Gamma(q)}|f(s,x(s))-f(s,y(s))|ds\\
&&+\sum_{i=1}^{n}|\beta_{i}|\int_{0}^{\eta_{i}}\frac{(\eta_{i}-s)^{q}}{\Gamma(q+1)}|f(s,x(s))-f(s,y(s))|ds \nonumber\\
&&+\sum_{i=1}^{n}|\gamma_{i}|\int_{0}^{\eta_{i}}\frac{(\eta_{i}-s)^{q-1}}{\Gamma(q)}|f(s,x(s))-f(s,y(s))|ds \Bigg)\\
&&+\frac{\Gamma(2-\sigma)\Gamma(2-\nu)\Big(|\Delta_{3}|+2|\Delta_{2}|\Big)}{2|\Delta_{1}\Delta_{2}|} \nonumber\\
&&\times \Bigg(\int_{0}^{\xi}\frac{(\xi-s)^{q-\sigma-1}}{\Gamma(q-\sigma)}|f(s,x(s))-f(s,y(s))|ds\nonumber \\
&&+\sum_{i=1}^{n}|\alpha_{i}|\int_{0}^{\eta_{i}}\frac{(\eta_{i}-s)^{q-\nu-1}}{\Gamma(q-\nu)}|f(s,x(s))-f(s,y(s))|ds \Bigg)\\
&\leq& L\|x-y\|\Bigg\{\frac{1}{\Gamma(q+1)}+\frac{1}{|\Delta_{2}|\Gamma(q+2)}\Bigg(1+q+\sum_{i=1}^{n}\eta_{i}^{q}\bigg(\eta_{i}|\beta_{i}|
+(q+1)|\gamma_{i}|\bigg)\Bigg)\\
&&+\frac{\Gamma(2-\sigma)\Gamma(2-\nu)\Big(|\Delta_{3}|+2|\Delta_{2}|\Big)}{2|\Delta_{1}\Delta_{2}|}
\times \Bigg(\frac{\xi^{q-\sigma}}{\Gamma(q-\sigma+1)}+\sum_{i=1}^{n}|\alpha_{i}|\frac{\eta_{i}^{q-\nu}}{\Gamma(q-\nu+1)} \Bigg)\Bigg\}\\
&=&L\Theta \|x-y\|.
\end{eqnarray*}
Consequently, $ \|\mathcal{A}x-\mathcal{A}y\|\leq L \Theta \|x-y\|$. As $L\Theta<1$, it follows that the operator $\mathcal{A}$ is a contraction. By the Banach's contraction mapping principle, $\mathcal{A}$ has a fixed point in $B_{\rho}$ which is the unique solution of the problem \eqref{eq-1.1}-\eqref{eq-1.2} on $[0,1]$. This completes the proof.
\end{proof}

Next, we give a second existence and uniqueness result based on nonlinear contractions.

\begin{definition} \label{def 3.1}
Let $E$ be a Banach space and let $A:E\rightarrow E$ be a mapping. $A$ is said to be a nonlinear contraction if there exists a continuous nondecreasing function $\psi:\mathbb{R}^{+}\rightarrow \mathbb{R}^{+}$ such that $\psi(0)=0$ and $\psi(\alpha)<\alpha$ for all $\alpha>0$ with the property:
\[\|Ax-Ay\|\leq\psi(\|x-y\|),\ \forall x,y \in E.\]
\end{definition}

\begin{lemma} [Boyd and Wong\cite{Boyd1969}]\label{lem 3.1}
Let $E$ be a Banach space and let $A:E\rightarrow E$ be a nonlinear contraction. Then, $A$ has a unique fixed point in $E $.
\end{lemma}

\begin{theorem} \label{thm 3.2}
Let $f:[0,1]\times\mathbb{R}\longrightarrow\mathbb{R}$ be a continuous function satisfying the assumption:
\newline
$(H_{2})$ $|f(t,x)-f(t,y)|\leq g(t)\Phi^{-1} \ln(1+|x-y|)$, for all $t\in[0, 1]$
and $x,y\in \mathbb{R}$, where $g:[0,1]\rightarrow \mathbb{R}^{+}$ is continuous and the positive constant $\Phi$ is defined by
\begin{eqnarray} \label{eq-3.4}
\Phi&=&\Big(1+\frac{1}{|\Delta_{2}|}\Big)\int_{0}^{1}\frac{(1-s)^{q-1}}{\Gamma(q)}g(s)ds
+\frac{1}{|\Delta_{2}|}\sum_{i=1}^{n}\int_{0}^{\eta_{i}}\frac{(\eta_{i}-s)^{q-1}}{\Gamma(q+1)}
\Big(q|\gamma_{i}|+|\beta_{i}|(\eta_{i}-s)\Big)g(s)ds\nonumber\\ &&+\frac{\Gamma(2-\sigma)\Gamma(2-\nu)\Big(|\Delta_{3}|+2|\Delta_{2}|\Big)}{2|\Delta_{1}\Delta_{2}|}\Bigg(
\int_{0}^{\xi}\frac{(\xi-s)^{q-\sigma-1}}{\Gamma(q-\sigma)}g(s)ds+\sum_{i=1}^{n}|\alpha_{i}|\int_{0}^{\eta_{i}}\frac{(\eta_{i}-s)^{q-\nu-1}}{\Gamma(q-\nu)}g(s)ds
\Bigg)\nonumber.
\end{eqnarray}
Then the fractional boundary value problem \eqref{eq-1.1}-\eqref{eq-1.2} has a unique solution on $[0,1]$.
\end{theorem}

\begin{proof}
Let the operator $\mathcal{A} : \mathcal{C}\rightarrow\mathcal{C}$ be defined as in \eqref{eq-3.1}. Consider the continuous non decreasing function $\psi:\mathbb{R}^{+}\rightarrow \mathbb{R}^{+}$ defined by
\[\psi(\alpha)=\ln(\alpha+1),\ \forall\alpha\geq0.\]
Clearly, the function $\psi$ satisfies $\psi(0)=0$ and $\psi(\alpha)<\alpha$, for all $\alpha>0$.

For any $x,y\in \mathcal{C}$ and for each $t\in [0,1]$, we have
\begin{eqnarray*}
|(\mathcal{A}x)(t)-(\mathcal{A}y)(t)|&\leq&\int_{0}^{t}\frac{(t-s)^{q-1}}{\Gamma(q)}|f(s,x(s))-f(s,y(s))|ds \nonumber\\
&&+\frac{1}{|\Delta_{2}|}\Bigg(\int_{0}^{1}\frac{(1-s)^{q-1}}{\Gamma(q)}|f(s,x(s))-f(s,y(s))|ds \\
&&+\sum_{i=1}^{n}|\beta_{i}|\int_{0}^{\eta_{i}}\frac{(\eta_{i}-s)^{q}}{\Gamma(q+1)}|f(s,x(s))-f(s,y(s))|ds \nonumber\\
&&+\sum_{i=1}^{n}|\gamma_{i}|\int_{0}^{\eta_{i}}\frac{(\eta_{i}-s)^{q-1}}{\Gamma(q)}|f(s,x(s))-f(s,y(s))|ds \Bigg)\\
&&+\frac{\Gamma(2-\sigma)\Gamma(2-\nu)\Big(|\Delta_{3}|+2|\Delta_{2}|\Big)}{2|\Delta_{1}\Delta_{2}|} \times \Bigg(\int_{0}^{\xi}\frac{(\xi-s)^{q-\sigma-1}}{\Gamma(q-\sigma)}|f(s,x(s))-f(s,y(s))|ds\nonumber\\
&&+\sum_{i=1}^{n}|\alpha_{i}|\int_{0}^{\eta_{i}}\frac{(\eta_{i}-s)^{q-\nu-1}}{\Gamma(q-\nu)}|f(s,x(s))-f(s,y(s))|ds \Bigg)\\
&\leq&\Phi^{-1}\int_{0}^{t}\frac{(t-s)^{q-1}}{\Gamma(q)}g(s)\ln\big(1+|x(s)-y(s)|\big)ds\\
&&+\frac{\Phi^{-1}}{|\Delta_{2}|}\Bigg(\int_{0}^{1}\frac{(1-s)^{q-1}}{\Gamma(q)}g(s)\ln\big(1+|x(s)-y(s)|\big)ds
\end{eqnarray*}
\begin{eqnarray*}
&&+\sum_{i=1}^{n}|\beta_{i}|\int_{0}^{\eta_{i}}\frac{(\eta_{i}-s)^{q}}{\Gamma(q+1)}g(s)\ln\big(1+|x(s)-y(s)|\big)ds \nonumber\\
&&+\sum_{i=1}^{n}|\gamma_{i}|\int_{0}^{\eta_{i}}\frac{(\eta_{i}-s)^{q-1}}{\Gamma(q)}g(s)\ln\big(1+|x(s)-y(s)|\big)ds\Bigg) \nonumber\\
&&+\frac{\Phi^{-1}\Gamma(2-\sigma)\Gamma(2-\nu)\Big(|\Delta_{3}|+2|\Delta_{2}|\Big)}{2|\Delta_{1}\Delta_{2}|} \nonumber\\
&&\times \Bigg(\int_{0}^{\xi}\frac{(\xi-s)^{q-\sigma-1}}{\Gamma(q-\sigma)}g(s)\ln\big(1+|x(s)-y(s)|\big)ds\nonumber \\
&&+\sum_{i=1}^{n}|\alpha_{i}|\int_{0}^{\eta_{i}}\frac{(\eta_{i}-s)^{q-\nu-1}}{\Gamma(q-\nu)}g(s)\ln\big(1+|x(s)-y(s)|\big)ds \Bigg)\nonumber\\
&\leq&\Phi^{-1}\psi(\|x-y\|)\Bigg\{\int_{0}^{1}\frac{(1-s)^{q-1}}{\Gamma(q)}g(s)ds+\frac{1}{|\Delta_{2}|}\Bigg(\int_{0}^{1}\frac{(1-s)^{q-1}}{\Gamma(q)}g(s)ds \nonumber\\
&&+\sum_{i=1}^{n}|\beta_{i}|\int_{0}^{\eta_{i}}\frac{(\eta_{i}-s)^{q}}{\Gamma(q+1)}g(s)ds +\sum_{i=1}^{n}|\gamma_{i}|\int_{0}^{\eta_{i}}\frac{(\eta_{i}-s)^{q-1}}{\Gamma(q)}g(s)ds\Bigg) \nonumber \\
&&+\frac{\Gamma(2-\sigma)\Gamma(2-\nu)\Big(|\Delta_{3}|+2|\Delta_{2}|\Big)}{2|\Delta_{1}\Delta_{2}|}\times \Bigg(\int_{0}^{\xi}\frac{(\xi-s)^{q-\sigma-1}}{\Gamma(q-\sigma)}g(s)ds\\
&&+\sum_{i=1}^{n}|\alpha_{i}|\int_{0}^{\eta_{i}}\frac{(\eta_{i}-s)^{q-\nu-1}}{\Gamma(q-\nu)}g(s)\big)ds \Bigg)\Bigg\}\\
&=& \Phi^{-1}\psi(\|x-y\|)\Phi\\
&=&\psi(\|x-y\|).
\end{eqnarray*}
Then, $\|\mathcal{A}x-\mathcal{A}y\|\leq \psi(\|x-y\|)$. Hence, $\mathcal{A}$ is a nonlinear contraction. Therefore, by Lemma \ref{lem 3.1}, the operator $\mathcal{A}$ has a unique fixed point in $\mathcal{C}$, which is a unique solution of problem \eqref{eq-1.1}-\eqref{eq-1.2}. This completes the proof.
\end{proof}
Our last existence result is based on Leray-Schauder nonlinear alternative \cite{Granas2003}.

\begin{lemma} [Nonlinear alternative for single valued maps \cite{Granas2003}]\label{lem 3.2}
Let $E$ be a Banach space, $E_{1}$ a closed, convex subset of $E$, $U$ an open subset of $E_{1}$ and $0\in U$. Suppose that $A:\bar{U}\rightarrow E_{1}$ is a continuous, compact (that is $A(\bar{U})$ is a relatively compact subset of $E_{1}$) map. Then either
\item[(i)] $A$ has a fixed point in $\bar{U}$, or
\item[(ii)] there is a $x\in \partial {U}$ (the boundary of $U$ in $E_{1}$) and $\lambda \in (0,1)$ with $x=\lambda A(x)$.
\end{lemma}

\begin{theorem} \label{thm 3.3}
Let $f:[0,1]\times\mathbb{R}\longrightarrow\mathbb{R}$ be a continuous function. Further, it is assumed that:
\newline
$(H_{3})$ there exist a function $p \in C([0,1],\mathbb{R^{+}})$ and a nondecreasing function $\psi:\mathbb{R}^{+}\rightarrow \mathbb{R}^{+}$ such that $|f(t,x)|\leq p(t) \psi(|x|)$, $\forall(t,x)\in [0,1]\times \mathbb{R}$;
\newline
$(H_{4})$ there exists a constant $M>0$ such that $\frac{M}{\Theta \psi(M)\|p\|}>1$, where $\Theta$ is given by \eqref{eq-3.2}.
\\
Then the fractional boundary value problem \eqref{eq-1.1}-\eqref{eq-1.2} has at least one solution on $[0,1]$.
\end{theorem}

\begin{proof}
Let consider the operator $\mathcal{A} : \mathcal{C}\rightarrow\mathcal{C}$ defined by \eqref{eq-3.1}. In view of the continuity of $f$, the operator $\mathcal{A}$ is continuous.
Firstly, we will show that the operator $\mathcal{A}$, maps bounded sets into bounded sets in $\mathcal{C}$. For a positive number $\rho$, let $\mathcal{B}_{\rho}=\{x\in \mathcal{C}:\|x\|\leq \rho\}$ be bounded set in $\mathcal{C}$. Then, for $t\in[0,1]$ and $x\in \mathcal{B}_{\rho}$ together with $(H_{3})$, we obtain
\begin{eqnarray*}
|(\mathcal{A}x)(t)|&=&\Bigg|\int_{0}^{t}\frac{(t-s)^{q-1}}{\Gamma(q)}f(s,x(s))ds+\frac{1}{\Delta_{2}}\Bigg(-\int_{0}^{1}\frac{(1-s)^{q-1}}{\Gamma(q)}f(s,x(s))ds\\ &&+\sum_{i=1}^{n}\beta_{i}\int_{0}^{\eta_{i}}\frac{(\eta_{i}-s)^{q}}{\Gamma(q+1)}f(s,x(s))ds+\sum_{i=1}^{n}\gamma_{i}\int_{0}^{\eta_{i}}\frac{(\eta_{i}-s)^{q-1}}{\Gamma(q)}f(s,x(s))ds\Bigg) \end{eqnarray*}
\begin{eqnarray*}
&&+\frac{\Gamma(2-\sigma)\Gamma(2-\nu)\Big(\Delta_{3}+2\Delta_{2}t\Big)}{2\Delta_{1}\Delta_{2}}\Bigg(-\int_{0}^{\xi}\frac{(\xi-s)^{q-\sigma-1}}{\Gamma(q-\sigma)}f(s,x(s))ds\nonumber \\
&&+\sum_{i=1}^{n}\alpha_{i}\int_{0}^{\eta_{i}}\frac{(\eta_{i}-s)^{q-\nu-1}}{\Gamma(q-\nu)}f(s,x(s))ds \Bigg)\Bigg|\\
&\leq&\int_{0}^{t}\frac{(t-s)^{q-1}}{\Gamma(q)}p(s)\psi(\|x\|)ds+\frac{1}{|\Delta_{2}|}\Bigg(\int_{0}^{1}\frac{(1-s)^{q-1}}{\Gamma(q)}p(s)\psi(\|x\|)ds\\
&&+\sum_{i=1}^{n}|\beta_{i}|\int_{0}^{\eta_{i}}\frac{(\eta_{i}-s)^{q}}{\Gamma(q+1)}p(s)\psi(\|x\|)ds+\sum_{i=1}^{n}|\gamma_{i}|\int_{0}^{\eta_{i}}\frac{(\eta_{i}-s)^{q-1}}{\Gamma(q)}p(s)\psi(\|x\|)ds\Bigg) \\
&&+\frac{\Gamma(2-\sigma)\Gamma(2-\nu)\Big(|\Delta_{3}|+2|\Delta_{2}|\Big)}{2|\Delta_{1}\Delta_{2}|}\Bigg(\int_{0}^{\xi}\frac{(\xi-s)^{q-\sigma-1}}{\Gamma(q-\sigma)}p(s)\psi(\|x\|)ds\nonumber \\
&&+\sum_{i=1}^{n}|\alpha_{i}|\int_{0}^{\eta_{i}}\frac{(\eta_{i}-s)^{q-\nu-1}}{\Gamma(q-\nu)}p(s)\psi(\|x\|)ds \Bigg)\\
&\leq&\Bigg\{\frac{t^{q}}{\Gamma(q+1)}+\frac{1}{|\Delta_{2}|}\Bigg(\frac{1}{\Gamma(q+1)}+\sum_{i=1}^{n}|\beta_{i}|\frac{\eta_{i}^{q+1}}{\Gamma(q+2)}
+\sum_{i=1}^{n}|\gamma_{i}|\frac{\eta_{i}^{q}}{\Gamma(q+1)}\Bigg)\nonumber\\
&&+\frac{\Gamma(2-\sigma)\Gamma(2-\nu)\Big(|\Delta_{3}|+2|\Delta_{2}|\Big)}{2|\Delta_{1}\Delta_{2}|}\Bigg(\frac{\xi^{q-\sigma}}{\Gamma(q-\sigma+1)}+\sum_{i=1}^{n}|\alpha_{i}|\frac{\eta_{i}^{q-\nu}}{\Gamma(q-\nu+1)} \Bigg)\Bigg\}\times \|p\|\psi(\rho)\\
&\leq&\Bigg\{\frac{1}{\Gamma(q+1)}+\frac{1}{|\Delta_{2}|\Gamma(q+2)}\Bigg(1+q+\sum_{i=1}^{n}\eta_{i}^{q}\bigg(\eta_{i}|\beta_{i}|+(q+1)|\gamma_{i}|\bigg)\Bigg)\nonumber\\
&&+\frac{\Gamma(2-\sigma)\Gamma(2-\nu)\Big(|\Delta_{3}|+2|\Delta_{2}|\Big)}{2|\Delta_{1}\Delta_{2}|}\Bigg(\frac{\xi^{q-\sigma}}{\Gamma(q-\sigma+1)}+\sum_{i=1}^{n}|\alpha_{i}|\frac{\eta_{i}^{q-\nu}}{\Gamma(q-\nu+1)} \Bigg)\Bigg\}\times\|p\|\psi(\rho)\\
&=&\Theta\|p\|\psi(\rho).
\end{eqnarray*}
Thus, we conclude that $\|\mathcal{A}x\|\leq \Theta\|p\|\psi(\rho)$.
This clearly validate the uniform boundedness of the set $\mathcal{A}(\mathcal{B}_{\rho})$.

Next, we show that the operator $\mathcal{A}$ maps bounded sets into equicontinuous sets of  $\mathcal{C}=C([0,1],\mathbb{R})$.

Set $f_{max}=\max_{(t,x)\in[0,1]\times \mathcal{B}_{\rho}}\{f(t,x)\}$. For $x\in \mathcal{B}_{\rho}$ and all $t \in[0,1]$, we obtain
\begin{eqnarray*}
|(\mathcal{A}x)'(t)|&=&\Bigg|\int_{0}^{t}\frac{(t-s)^{q-2}}{\Gamma(q-1)}f(s,x(s))ds
+\frac{\Gamma(2-\sigma)\Gamma(2-\nu)}{\Delta_{1}}\Bigg(-\int_{0}^{\xi}\frac{(\xi-s)^{q-\sigma-1}}{\Gamma(q-\sigma)}f(s,x(s))ds\nonumber \\
&&+\sum_{i=1}^{n}\alpha_{i}\int_{0}^{\eta_{i}}\frac{(\eta_{i}-s)^{q-\nu-1}}{\Gamma(q-\nu)}f(s,x(s))ds \Bigg)\Bigg| \\
&\leq&f_{max}\Bigg\{\int_{0}^{t}\frac{(t-s)^{q-2}}{\Gamma(q-1)}ds+\frac{\Gamma(2-\sigma)\Gamma(2-\nu)}{|\Delta_{1}|}\\
&&\times\Bigg(\int_{0}^{\xi}\frac{(\xi-s)^{q-\sigma-1}}{\Gamma(q-\sigma)}ds+\sum_{i=1}^{n}|\alpha_{i}|\int_{0}^{\eta_{i}}\frac{(\eta_{i}-s)^{q-\nu-1}}{\Gamma(q-\nu)}ds \Bigg)\Bigg\}\\
&\leq& \Omega  f_{max},
\end{eqnarray*}
where $\Omega$ is given by \eqref{eq-3.3}.
\\
Hence, for $t_{1}, t_{2}\in [0,1]$ with $t_{1}<t_{2}$, it follows that
\[|(\mathcal{A}x)(t_{1})-(\mathcal{A}x)(t_{2})|\leq\int_{t_{1}}^{t_{2}}|(\mathcal{A}x)'(s)|ds\leq \Omega  f_{max}(t_{2}-t_{1}). \]
So, we have shown that the set $\mathcal{A}(\mathcal{B}_{\rho})$ is equicontinuous in $\mathcal{C}$. Thus, by Arzel\'{a}-Ascoli theorem, we conclude that $\mathcal{A}$ is completely continuous operator.
\\
Finally, it is to show there exists an open $U\subset \mathcal{C}$ with $x\neq \lambda \mathcal{A}(x)$ for any $\lambda \in(0,1)$ and all $x\in\partial U$.
\\
Define $U=\{x\in \mathcal{C}:\|x\|< M\}$, and assume that there exists $x\in\partial U$ such that $x=\lambda \mathcal{A}(x)$ for some $\lambda \in(0,1)$.

Note that the operator $\mathcal{A} : \bar{U}\rightarrow\mathcal{C}$ is continuous and completely continuous.
Using \eqref{eq-3.1}, as before for all $t \in[0,1]$ , we obtain
\[|x(t)|=\big|\lambda(\mathcal{A}x)(t)\big|\leq \Theta\|p\|\psi(\|x\|),\]
which, on taking the norm for $t \in[0,1]$, yields
\[\frac{\|x\|}{\Theta \psi(\|x\|)\|p\|}\leq1,\]
actually contradicts the condition $(H_{4})$. Hence $x\neq \lambda \mathcal{A}(x)$ for $x\in\partial U$, $\lambda \in(0,1)$.
Consequently, by Lemma \ref{lem 3.2}, we deduce that the operator $\mathcal{A}$ has a fixed point in $\bar{U}$, which is a desired solution of the problem \eqref{eq-1.1}-\eqref{eq-1.2}. This completes the proof.
\end{proof}
\section{Examples\label{sec4}}
\begin{exmp}
Consider the following nonlocal fractional boundary value problem

\begin{equation} \label{eq-4.1}
\begin{cases}
^{c}D^{\frac{3}{2}}x(t)=\frac{te^{-\pi{t}}}{56+e^{-2t}}\sin{x}+\frac{e^{-\cos^{2}t}}{\sqrt{64+t}}\tan^{-1}x+\frac{1}{3},\ t\in[0,1],
\\
^{c}D^{\frac{1}{3}}x(\frac{3}{5})=\ ^{c}D^{\frac{1}{4}}x(\frac{4}{5})+\frac{1}{2}\ ^{c}D^{\frac{1}{4}}x(\frac{6}{7}),
\\ x(1)=\frac{1}{3}\int_{0}^{\frac{4}{5}}x(s)ds+\frac{2}{3}\int_{0}^{\frac{6}{7}}x(s)ds+3x(\frac{4}{5})+\frac{1}{7}x(\frac{6}{7}),
\end{cases}
\end{equation}

where $q=\frac{3}{2}$, $\nu=\frac{1}{4}$, $\sigma=\frac{1}{3}$, $\eta_{1}=\frac{4}{5}$, $\eta_{2}=\frac{6}{7}$, $\xi=\frac{3}{5}$, $\alpha_{1}=1$, $\alpha_{2}=\frac{1}{2}$, $\beta_{1}=\frac{1}{3}$, $\beta_{2}=\frac{2}{3}$, $\gamma_{1}=3$, $\gamma_{2}=\frac{1}{7}$, and
$f(t,x)=\frac{te^{-\pi{t}}}{56+e^{-2t}}\sin{x}+\frac{e^{-\cos^{2}t}}{\sqrt{64+t}}\tan^{-1}x+\frac{1}{3}$.
\\
Since

\begin{eqnarray*}
|f(t,x)-f(t,y)|&\leq&\frac{te^{-\pi{t}}}{56+e^{-2t}}|\sin{x}-\sin{y}|+\frac{e^{-\cos^{2}t}}{\sqrt{64+t}}|\tan^{-1}{x}-\tan^{-1}y|\\
&\leq&\frac{1}{56}|x-y|+\frac{1}{8}|x-y|\\
&=&\frac{1}{7}|x-y|,
\end{eqnarray*}
then $(H_{1})$ is satisfied with $L=\frac{1}{7}$.
\\
Using the given data, we find that $\Delta_{1}=-0.51192$, $\Delta_{2}=-\frac{313}{105}$, $\Delta_{3}=\frac{13774}{3675}$, $\Theta=5.70719$. Hence, we get $L \Theta \simeq 0.81531<1$.

Thus all the conditions of Theorem \ref{thm 3.1} are satisfied. So, the fractional boundary value problem \eqref{eq-4.1} has a unique solution on $[0,1]$.

\end{exmp}

\begin{exmp}
Consider a fractional boundary value problem given by

\begin{equation} \label{eq-4.2}
\begin{cases}
^{c}D^{\frac{7}{6}}x(t)=\frac{1}{11}e^{-t}\Big(\frac{2x^{3}}{1+x^{2}}+\frac{7+t}{2(5+\cos{t})}+1\Big),\ t\in[0,1],
\\
^{c}D^{\frac{1}{2}}x(\frac{1}{5})=2\ ^{c}D^{\frac{1}{3}}x(\frac{1}{4})+3\ ^{c}D^{\frac{1}{3}}x(\frac{2}{3}),
\\
x(1)=\frac{2}{5}\int_{0}^{\frac{1}{4}}x(s)ds+\frac{1}{7}\int_{0}^{\frac{2}{3}}x(s)ds+\frac{1}{2}x(\frac{1}{4})+x(\frac{2}{3}),
\end{cases}
\end{equation}

where $q=\frac{7}{6}$, $\nu=\frac{1}{3}$, $\sigma=\frac{1}{2}$, $\eta_{1}=\frac{1}{4}$, $\eta_{2}=\frac{2}{3}$, $\xi=\frac{1}{5}$, $\alpha_{1}=2$, $\alpha_{2}=3$, $\beta_{1}=\frac{2}{5}$, $\beta_{2}=\frac{1}{7}$, $\gamma_{1}=\frac{1}{2}$, $\gamma_{2}=1$, and
$f(t,x)=\frac{1}{11}e^{-t}\Big(\frac{2x^{3}}{1+x^{2}}+\frac{7+t}{2(5+\cos{t})}+1\Big)$.
\\
A simple computation gives
\[ \Delta_{1}=-2.32863,\ \Delta_{2}=-\frac{73}{105},\ \Delta_{3}=-\frac{827}{2520},\ \Theta=4.67261.\]

Clearly,
$$|f(t,x)|=\Bigg|\frac{1}{11}e^{-t}\Big(\frac{2x^{3}}{1+x^{2}}+\frac{7+t}{2(5+\cos{t})}+1\Big)\Bigg| \leq \frac{2}{11}e^{-t}(|x|+1)$$.

Choosing $p(t)=\frac{2}{11}e^{-t}$ and $\psi(|x|)=|x|+1$. Then, we have
$$\frac{M}{\Theta \psi(M)\|p\|}=\frac{11M}{2(4.67261)(M+1)}>1,$$

which implies that there exists a constant $M>5.64742$. Hence, by Theorem \ref{thm 3.3},
the nonlocal boundary value problem \eqref{eq-4.2} has at least one solution on $[0, 1]$.

\end{exmp}

\begin{exmp}
As a third example we consider the fractional boundary value problem
\begin{equation} \label{eq-4.3}
\begin{cases}
^{c}D^{\frac{4}{3}}x(t)=\frac{1}{6}e^{-t^{2}} \ln(1+|x|),\ t\in[0,1],
\\
^{c}D^{\frac{4}{5}}x(\frac{3}{11})= \frac{3}{7} \ ^{c}D^{\frac{2}{3}}x(\frac{7}{8})+ \frac{11}{12} \ ^{c}D^{\frac{2}{3}}x(\frac{8}{9}),
\\
x(1)= \frac{1}{4}\int_{0}^{\frac{7}{8}}x(s)ds+\frac{3}{2}\int_{0}^{\frac{8}{9}}x(s)ds+\frac{1}{10}x(\frac{7}{8})+\frac{2}{5}x(\frac{8}{9}),
\end{cases}
\end{equation}
where
$q=\frac{4}{3}$, $\nu=\frac{2}{3}$, $\sigma=\frac{4}{5}$, $\eta_{1}=\frac{7}{8}$, $\eta_{2}=\frac{8}{9}$, $\xi=\frac{3}{11}$, $\alpha_{1}=\frac{3}{7}$, $\alpha_{2}=\frac{11}{12}$, $\beta_{1}=\frac{1}{4}$, $\beta_{2}=\frac{3}{2}$, $\gamma_{1}=\frac{1}{10}$, $\gamma_{2}=\frac{2}{5}$, and
$f(t,x)=\frac{1}{6}e^{-t^{2}} \ln(1+|x|)$.
\\
With the given values, it is found that
\[ \Delta_{1}=-0.496989,\ \Delta_{2}=-\frac{101}{96},\ \Delta_{3}=\frac{9079}{34560}.\]
\\
By choosing $g(t)=\frac{1}{6}e^{-t^2}$, we find that $\Phi=0.809777$.
\\
Obviously,
\begin{eqnarray*}
|f(t,x)-f(t,y)|&=&\Big|g(t)\Big(\ln(1+|x|)-\ln(1+|y|)\Big)\Big|\\
&=&g(t)\Bigg|\ln\Bigg(\frac{1+|x|}{1+|y|}\Bigg)\Bigg|\\
&\leq&g(t)\ln(1+|x-y|)\\
&\leq&g(t)\Phi^{-1}\ln(1+|x-y|).
\end{eqnarray*}

Hence, by Theorem \ref{thm 3.2}, the nonlocal boundary value problem \eqref{eq-4.3} has a unique solution
on $[0, 1]$.
\end{exmp}


\section*{Acknowledgments}





\begin{thebibliography}{}
\bibitem{Agarwal(2)2017}
R. P. Agarwal, A. Alsaedi, A. Alsharif, B. Ahmad
, \textit{On nonlinear fractional-order boundary value problems with nonlocal multi-point conditions involving Liouville-Caputo derivatives}, Differ. Equ. Appl., \textbf{9}(2) (2017), 147–-160, doi:10.7153/dea-09-12.



\bibitem{Agarwal2017}
R. P. Agarwal, B. Ahmad, D. Garout, A. Alsaedi, \textit{Existence results for coupled nonlinear fractional differential equations equipped with nonlocal coupled flux and multi-point boundary conditions}, Chaos, Solitons and Fractals, \textbf{102} (2017), 149--161


\bibitem{Ahmad2015}
B. Ahmad, A. Alsaedi, A. Alsharif, \textit{Existence result for fractional-order differential equations with nonlocal multi-point-strip conditions involving Caputo derivative}, Adv. Differ. Equ., \textbf{348} (2015).

\bibitem{Ahmad2016}
B. Ahmad, S. K. Ntouyas, R. P. Agarwal, A. Alsaedi, \textit{Existence results for sequential fractional
integro-differential equations with nonlocal multi-point and strip conditions}, Bound. Value Probl., \textbf{205} (2016).

\bibitem{Ahmad2016(2)}
B. Ahmad, A. Alsaedi, D. Garout, \textit{Existence results for Liouville-Caputo type fractional differential
equations with nonlocal multi-point and sub-strips boundary conditions}, Comput.Math. Appl., (2016).
https://doi.org/10.1016/j.camwa.2016.04.015.

\bibitem{Boyd1969}
D. W. Boyd, J. S. W. Wong, \textit{On nonlinear contractions}. Proc. Amer. Math. Soc. 20 (1969), 458–-464.


\bibitem{Di2018}
D. Bin, P. Huihui, \textit{Existence results for the fractional differential
equations with multi-strip integral boundary conditions}, J. Appl. Math. Comput.,
https://doi.org/10.1007/s12190-018-1166-z. pp 1-19.



\bibitem{Granas2003}
A. Granas, J. Dugundji, \textit{Fixed point theory}, Springer-Verlag, New York, 2003. ISBN 0-387-00173-5.


\bibitem{Guo2008}
Y. Guo, Y. Ji, X. Liu, \textit{Multiple positive solutions for some multi-point boundary value
problems with p-Laplacian}, J. Comput. Appl. Math., \textbf{216} (1) (2008), 144–-156.




\bibitem{Haddouchi2018}
F. Haddouchi, \textit{Existence results for a class of Caputo type fractional differential equations with Riemann-Liouville fractional integrals and Caputo fractional derivatives in boundary conditions}, [https://arxiv.org/abs/1805.06015], May, 2018.

\bibitem{Henderson2017}
J. Henderson, R. Luca, \textit{Systems of Riemann-Liouville fractional equations with multi-point boundary conditions}, Appl. Math. Comput. \textbf{309} (2017), 303--323.


\bibitem{Jia2012}
M. Jia, X. Zhang, X, Gu, \textit{ Nontrivial solutions for a higher fractional differential equation with fractional multi-point boundary conditions}, Bound. Value Probl. 2012, 2012:70, 16 pp. https://doi.org/10.1186/1687-2770-2012-70.



\bibitem{Kilbas2006}
A. A. Kilbas, H. M. Srivastava, J. J. Trujillo,\textit{Theory and Applications of Fractional
Differential Equations}, North-Holland Mathematics Studies, vol. 204. Elsevier, Amsterdam, 2006. ISBN  978-0-444-51832-3; 0-444-51832-0.



\bibitem{Li2017}
Y. Li, A. Qi, \textit{ Existence of positive solutions for multi-point boundary value problems of Caputo fractional differential equation}, Int. J. Dyn. Syst. Differ. Equ., \textbf{7}(2) (2017), 169--183.



\bibitem{Liu2012}
Y. Liu, \textit{Solvability of multi-point boundary value problems for multiple term Riemann–Liouville fractional differential equations}, Comput. Math. Appl., \textbf{64}(4) (2012), 413--431.


\bibitem{Meral2010}
F. C. Meral, T. J. Royston, R. Magin, \textit{Fractional calculus in viscoelasticity: an experimental
study}, Commun. Nonlinear Sci. Numer. Simul.,  \textbf{15}(4) (2010), 939--945.


\bibitem{Miller1993}
K. S. Miller, B. Ross, \textit{An Introduction to the Fractional Calculus and Fractional Differential Equations}, Wiley, New York, 1993. ISBN  0-471-58884-9.


\bibitem{Nigmatullin2010}
R. Nigmatullin, T. Omay, D. Baleanu, \textit{On fractional filtering versus conventional filtering in
economics}, Commun. Nonlinear Sci. Numer. Simul., \textbf{15}(4) (2010), 979--986.


\bibitem{Oldham2010}
K. B. Oldham, \textit{Fractional differential equations in electrochemistry}, Adv. Eng. Softw., \textbf{41}(1) (2010), 9--12.

\bibitem{Oldham1974}
K. B. Oldham, J. Spanier, \textit{The Fractional Calculus}, Academic Press, New York, 1974.


\bibitem{Orsingher2004}
E. Orsingher, L. Beghin, \textit{Time-fractional telegraph equations and telegraph processes with
brownian time}, Probab. Theory. Related. Fields., \textbf{128}(1) (2004), 141-160.

\bibitem{Podlubny1999}
I. Podlubny, \textit{Fractional Differential Equations}, Academic Press, Inc., San Diego, 1999. ISBN  0-12-558840-2


\bibitem{Sun2014}
 Y. P. Sun, M. Zhao, \textit{Positive solutions for a class of fractional differential equations with integral boundary conditions}, Appl. Math. Lett., \textbf{34} (2014), 17-–21.


\bibitem{Wang2018}
Y. Wang, S. Liang, Q Wang, \textit{ Existence results for fractional differential equations with integral and multi-point boundary conditions},  Bound. Value Probl., 2018, Paper no. 4, 11 pp. https://doi.org/10.1186/s13661-017-0924-4.


\bibitem{Yang2016}
Y. Y. Yang, Q. R . Wang,  \textit{Positive solutions of multi-point boundary value problems}, Electron. J. Differ. Equ., \textbf{231} (2016).



\bibitem{Zhang2017}
Y. Zhang, Y. Gu, \textit{Eigenvalue intervals for nonlocal fractional order differential equations involving derivatives}, J. Appl. Math. Comput., \textbf{55} (2017), 119-–134 .


\bibitem{Zhang2012}
X. Zhang, L. Liu, Y. Wu, \textit{The eigenvalue problem for a singular higher order fractional differential equation involving fractional derivatives}, Appl. Math. Comput., \textbf{218} (2012), 8526–-8536.

\bibitem{Zhang2013}
X. Zhang, L. Liu, Y. Wu, \textit{The uniqueness of positive solution for a singular fractional differential system involving derivatives. Commun. Nonlinear Sci}, Numer. Simul., \textbf{18} (2013), 1400–-1409.

\bibitem{Zhang(2)2012}
X. Zhang, L. Liu, B. Wiwatanapataphee, Y. Wu, \textit{Solutions of eigenvalue problems for a class of fractional differential equations with derivatives}, Abstr. Appl. Anal., \textbf{2012} (2012), Article ID 512127.



\bibitem{Zhao2015}
K. Zhao, P. Gong, \textit{Positive solutions of m-point multi-term fractional integral BVP involving time-delay for fractional differential equations}, Bound. Value Probl., \textbf{19} (2015).

\bibitem{Zhao(2)2015}
K. Zhao, \textit{Triple positive solutions for two classes of delayed nonlinear fractional FDEs with nonlinear integral boundary value conditions}, Bound. Value Probl., \textbf{181} (2015).

\bibitem{Zhao2016}
K. Zhao, K. Wang, \textit{Existence of solutions for the delayed nonlinear fractional functional differential equations with three-point integral boundary value conditions}, Adv. Differ. Equ., \textbf{284} (2016).

\bibitem{Zhao2017}
K. Zhao, J. Liang, \textit{Solvability of triple-point integral boundary value problems for a class of impulsive fractional differential equations}, Adv. Differ. Equ., \textbf{50} (2017).

\bibitem{Zhao(2)2016}
K. Zhao, K. Wang, \textit{Existence of solutions for the delayed nonlinear fractional functional differential equations with three-point integral boundary value conditions}, Adv. Differ. Equ., \textbf{284} (2016).

\end{thebibliography}
\end{document}